\documentclass[12pt]{amsart}
\usepackage{amssymb,latexsym}
\usepackage{enumerate}

\newtheorem{thm}{Theorem}[section]
\newtheorem{cor}[thm]{Corollary}

\theoremstyle{definition}

\newcommand{\bC}{\mathbb{C}}

\newcommand{\cC}{\mathcal{C}}
\newcommand{\cM}{\mathcal{M}}

\begin{document}

\baselineskip=17pt


\title{Once More on Positive Commutators\footnote{The paper will appear in Studia Mathematica.}}

\author{Roman Drnov\v sek}

\address{Department of Mathematics,
    Faculty of Mathematics and Physics \\
    University of Ljubljana \\
    Jadranska 19, SI-1000 Ljubljana, Slovenia}

\email{roman.drnovsek@fmf.uni-lj.si}


\begin{abstract}
Let $A$ and $B$ be bounded operators on a Banach lattice $E$ such that the commutator $C = A B - B A$ and the product $BA$ are positive operators.
If the product $AB$ is a power-compact operator, then $C$ is a quasi-nilpotent operator having a triangularizing chain of closed ideals of $E$.
This theorem answers an open question posed in \cite{BDFRZ}, where the study of positive commutators of positive operators has been initiated.
\end{abstract}



\maketitle

\section{Introduction}

Let $X$ be a Banach space. 
The spectrum and the spectral radius of a bounded operator $T$ on $X$ are denoted by $\sigma(T)$ and $r(T)$,
respectively.  A bounded operator $T$ on $X$ is said to be {\it power-compact} if $T^n$ is a compact operator for some $n \in \mathbb{N}$.
A {\it chain} $\cC$ is a family of closed subspaces of $X$ that is totally ordered by inclusion. 
We say that $\cC$ is a {\it complete} chain if it is closed under arbitrary intersections and
closed linear spans. If $\cM$ is in a complete chain $\cC$, then the {\it predecessor}
$\cM_{-}$ of $\cM$ in $\cC$ is defined as the closed linear span of
all proper subspaces of $\cM$ belonging to $\cC$. 

Let $E$ be a Banach lattice. An operator $T$ on $E$ is called {\it positive} if the positive cone $E^+$ is invariant under $T$.  
It is well-known that every positive operator $T$ is bounded and that $r(T)$ belongs to $\sigma(T)$.
A bounded operator $T$ on $E$ is said to be {\it ideal-reducible} if there exists a non-trivial closed ideal of $E$
invariant under $T$. Otherwise, it is {\it ideal-irreducible}.
If the chain $\cC$ of closed ideals of $E$ is maximal in the lattice of all closed ideals of $E$ 
and if every one of its members is invariant under an operator $T$ on $E$, 
then $\cC$ is called a {\it triangularizing chain} for $T$, and $T$ is said to be {\it ideal-triangularizable}. 
Note that such a chain is also maximal in the lattice of all closed subspaces of $E$ (see e.g. \cite[Proposition 1.2]{Drn}).

In \cite{BDFRZ} positive commutators of positive operators on Banach lattices are studied.
The main result \cite[Theorem 2.2]{BDFRZ} is the following

\begin{thm}
\label{compact} 
Let $A$ and $B$ be positive compact operators on a Banach lattice $E$ 
such that the commutator $C = A B - B A$ is also positive. 
Then $C$ is an ideal-triangularizable quasi-nilpotent operator.
\end{thm}

Examples in \cite{BDFRZ} show that the compactness assumption of Theorem \ref{compact} cannot be omitted. 
They are based on a simple example that can be obtained by setting $A = S^*$ and $B = S$, 
where $S$ is the unilateral shift on the Banach lattice $l^2$.

Theorem \ref{compact} has been futher extended in \cite[Theorem 3.4]{DK}. Recall that a bounded operator $T$ on 
a Banach space is called a {\it Riesz operator} or an {\it essentially quasi-nilpotent operator} if $\{0\}$ is the essential spectrum of $T$.

\begin{thm}
\label{sum_Riesz}
Let $A$ and $B$ be positive operators on a Banach lattice $E$ such that the sum $A + B$ is 
a Riesz operator. If the commutator $C = A B - B A$ is a power-compact positive operator, 
then it is an ideal-triangularizable quasi-nilpotent operator.
\end{thm}

In this note we answer affirmatively the open question posed in \cite[Open questions 3.7 (1)]{BDFRZ}
whether is it enough to assume in Theorem \ref{compact} that only one of the
operators $A$ and $B$ is compact.

\section{Preliminaries}

If $T$ is a power-compact operator on a Banach space $X$, then, by the classical spectral theory,  for each $\lambda \in \bC \setminus \{0\}$ 
the operator $\lambda - T$ has finite ascent $k$, i.e.,  $k$ is the smallest natural number such that 
${\rm ker \,} ((\lambda - T)^k) = {\rm ker \,} ((\lambda - T)^{k+1})$.
In this case the {\it (algebraic) multiplicity} $m(T, \lambda)$ of $\lambda$ is the dimension of the subspace 
${\rm ker \,} ((\lambda - T)^k)$. 

We will make use of the following extension of Ringrose's Theorem. 

\begin{thm}
\label{Ringrose} 
Let $T$ be a power-compact operator on a Banach space $X$, and let $\cC$ be a complete chain of closed subspaces
invariant under $T$.  Let $\cC^\prime$ be a subchain of $\cC$ of all subspaces $\cM \in \cC$ such that $\cM_{-} \neq \cM$.
For each $\cM \in \cC^\prime$, define $T_{\cM}$ to be the quotient operator
on $\cM / \cM_{-}$ induced by $T$. Then 
$$ \sigma(T) \setminus \{0\} =  \bigcup_{\cM \in \cC^\prime} \sigma(T_{\cM}) \setminus \{0\}. $$
Moreover, for each $\lambda \in \bC \setminus \{0\}$ we have 
$$ m(T, \lambda) = \sum_{\cM \in \cC^\prime} m(T_{\cM}, \lambda) . $$
\end{thm}

\begin{proof}
In the case of a compact operator $T$ the first equality is proved in \cite[Theorem 7.2.7]{RaRo}, while the second equality follows from 
the theorem \cite[Theorem 7.2.9]{RaRo} asserting that the algebraic multiplicity of each nonzero eigenvalue of $T$ is equal 
to its diagonal multiplicity with respect to any triangularizing chain.

An inspection of the proofs of these theorems reveals that it is enough to assume that the operator $T$ is power-compact.
Moreover, in \cite{Ko} the first equality was extended even to the case of polynomially compact operators.
\end{proof}

We will also need Pietsch's principle of related operators (see \cite[3.3.3]{Pi}).

\begin{thm}
\label{related_operators}
Let $A$ and $B$ be bounded operators on a Banach space. If $A B$ is power-compact, 
then $B A$ is power-compact and 
$$ m(A B, \lambda) = m(B A, \lambda) $$
for each $\lambda \in \bC \setminus \{0\}$.
\end{thm}

The following theorem is a consequence of \cite[Theorem 4.3]{Ma}; see a recent paper \cite[Theorem 0.1]{HKO}
which also contains the easily proved proposition \cite[Proposition 0.2]{HKO} that 
a positive operator is ideal-irreducible if and only if  it is semi non-supporting (the notion used in  \cite{Ma}).

\begin{thm}
\label{equal}
Let $S$ and $T$ be positive operators on a Banach lattice $E$ such that $S \le T$ and $r(S) = r(T)$. 
If $T$ is an ideal-irreducible power-compact operator, then $S = T$.
\end{thm}

\section{Results}

The main result of this note is the following extension of Theorem \ref{compact} (and \cite[Theorem 2.4]{BDFRZ} as well).

\begin{thm}
 \label{main}
Let $A$ and $B$ be bounded operators on a Banach lattice $E$ such that $AB \ge  BA \ge 0$ and $AB$ is a power-compact operator.
Then  the commutator $C = A B - B A$ is an ideal-triangularizable quasi-nilpotent operator.
\end{thm}

\begin{proof}
Let $\cC$ be a chain (of closed ideals) that is maximal in the lattice of all closed ideals
invariant under $A B$. By maximality, this chain is complete. 
 Let $\cC^\prime$ be a subchain of all subspaces $\cM \in \cC$ such that $\cM_{-} \neq \cM$.
Since $A B \ge B A \ge 0$ and $A B \ge C \ge 0$, every member of $\cC$ is also invariant under the operators $BA$ and $C$, and these operators are 
power-compact operators by the Aliprantis-Burkinshaw theorem \cite[Theorem 5.14]{AlBu}.
For any ideal $\cM \in \cC^\prime$, $r((A B)_{\cM}) \ge r((B A)_{\cM})$, since $(A B)_{\cM} \ge (B A)_{\cM} \ge  0$. 
We will prove that $r((A B)_{\cM}) = r((B A)_{\cM})$ for every ideal $\cM \in \cC^\prime$, and so 
$(A B)_{\cM} = (B A)_{\cM}$ by Theorem \ref{equal}.

Assume there are ideals  $\cM \in \cC^\prime$ such that $r((A B)_{\cM}) > r((B A)_{\cM})$. 
Among them choose  $\cM_0 \in \cC^\prime$  for which $\lambda_0 := r((A B)_{\cM_0})$ is maximal. Such an ideal exists, because  
for each $\epsilon > 0$ there are only finitely many eigenvalues of $A B$ with the absolute value at least $\epsilon$.
For each ideal $\cM \in \cC^\prime$ with $r((A B)_{\cM}) > \lambda_0$, we must have  $r((A B)_{\cM}) = r((B A)_{\cM})$, and so 
$(A B)_{\cM} = (B A)_{\cM}$ by Theorem \ref{equal}. The same conclusion holds in the case when $r((A B)_{\cM}) = r((B A)_{\cM}) = \lambda_0$.
If  $ \lambda_0 = r((A B)_{\cM}) > r((B A)_{\cM})$, then  
$$ m((A B)_{\cM}, \lambda_0) > 0 = m((B A)_{\cM}, \lambda_0). $$
If  $r((A B)_{\cM}) < \lambda_0$, then  
$$ m((A B)_{\cM}, \lambda_0) = 0 =  m((B A)_{\cM}, \lambda_0). $$
In view of Theorem \ref{Ringrose} we now conclude that  
$m(AB, \lambda_0) > m(BA, \lambda_0)$. However, by Theorem \ref{related_operators}, we have 
$m(AB, \lambda_0) = m(BA, \lambda_0)$.  
This contradiction shows that, for each $\cM \in \cC^\prime$,
$(A B)_{\cM} = (B A)_{\cM}$ and so $C_{\cM} = (A B)_{\cM} - (B A)_{\cM} = 0$.
By Theorem \ref{Ringrose}, we conclude that $C$ is quasi-nilpotent.

Finally, it is a simple consequence (see e.g. \cite[Theorem 1.3]{DK}) of the well-known de Pagter's theorem 
(see \cite[Theorem 9.19]{AbAl} or \cite{Pa86}) that $C$ has a triangularizing chain of closed ideals of $E$.
In fact, we can simply complete the chain $\cC$ to a triangularizing chain of closed ideals for the operator $C$. 
\end{proof}

As a corollary we obtain the answer to an open question posed in \cite[Open questions 3.7 (1)]{BDFRZ}.

\begin{cor}
\label{open}
Let $A$ and $B$ be positive operators on a Banach lattice $E$ such that the commutator $C = A B - B A$ is a positive operator.
If one of the operators  $A$ and $B$ is power-compact (in particular, compact), 
then the commutator $C$ is an ideal-triangularizable quasi-nilpotent operator.
\end{cor}

\begin{proof}
By a simple induction, we have $0 \le (A B)^n \le A^n B^n$  for every  $n \in \mathbb{N}$.
Assume now that for $n \in \mathbb{N}$ one of the operators $A^n$ and $B^n$ is compact, so that 
the operator $A^n B^n$ is compact. Then the operator $(AB)^{3n}$ is also compact by the
Aliprantis-Burkinshaw theorem \cite[Theorem 5.14]{AlBu}.
Therefore,  Theorem \ref{main} can be applied.
\end{proof}

It should be noted that a recent preprint \cite[Theorem 4.5]{Ga} gives an independent proof 
of Corollary \ref{open} in the case when one of the operators  $A$ and $B$ is compact.

\subsection*{Acknowledgements}
This research was partly supported by the Slovenian Research Agency.

\end{document}